\documentclass[a4paper,10pt]{elsarticle}

\usepackage[english]{babel}
\usepackage{amsmath,amsfonts,amssymb,latexsym,graphics,amsthm}
\usepackage[bookmarks=false]{hyperref}
\usepackage[dvips]{color}
\usepackage{pstricks}

\renewenvironment{itemize}{\begin{list}{\labelitemi}{\leftmargin=1.5em}}{\end{list}}

\def\Le{\hbox{\rotatedown{$\Gamma$}}}

\newcommand\pattern[4]{$\genfrac{}{}{0pt}{1}{#1#2}{#3#4}$}

\title{Bijections between pattern-avoiding fillings of Young diagrams}
\date{\today}
\author{Matthieu Josuat-Verg\`es}
\address{LRI, CNRS and Universit\'e Paris-Sud, 
  B\^atiment 490, 91405 Orsay CEDEX, FRANCE}

\newtheorem{prop}{Proposition}
\newtheorem{lem}{Lemma}
\newtheorem{definition}{Definition}

\newtheorem{corollary}{Corollary}

\begin{document}

\renewcommand{\labelitemi}{$\bullet$}

\begin{abstract}
The pattern-avoiding fillings of Young diagrams we study arose from
Postnikov's work on positive Grassman cells. They are called
\Le-diagrams, and are in bijection with decorated permutations.
Other closely-related diagrams are interpreted as acyclic orientations 
of some bipartite graphs. The definition of the diagrams is the same but the
avoided patterns are different. We give here bijections proving that
the number of pattern-avoiding filling of a Young diagram is the same,
for these two different sets of patterns. The result was obtained
by Postnikov via a reccurence relation. This relation was extended by
Spiridonov to obtain more general results about other patterns and
other polyominoes than Young diagrams, and we show that our bijections 
also extend to more general polyominoes.
\end{abstract}

\maketitle


\section{Introduction}

In his work on the combinatorics of the totally positive part of
the Grassmanian and its cell decomposition, A. Postnikov \cite{AP} 
introduced some diagrams called \Le-diagrams. 

\begin{definition}
Let $\lambda$ be a Young diagram, in English notation. 
A \Le-diagram $T$ of shape $\lambda$ 
is defined as a filling of every entry of $\lambda$ with a 0 or a 1,
such that for any 0 in $T$, all entries to its left (in the same row) 
are 0s, or all entries above it (in the same column) are 0s. 
\end{definition}

This definition is
equivalent to the following pattern-avoidance condition: a
diagram filled with 0s and 1s is a \Le-diagram, if and only if
it avoids the patterns \pattern 1110 and \pattern 0110. The
definition of "pattern-avoidance" is the following:

\begin{definition} Let $\lambda$ be a Young diagram. We call 
{\it diagram} of shape $\lambda$ a filling of entries of $\lambda$
with 0s and 1s.
For every square matrix $M$ of size 2, we say that a diagram $D$ of shape 
$\lambda$ avoids the pattern $M$ if there is no submatrix of
$D$ equal to $M$.
\end{definition}

In \cite{AP}, Postnikov gives bijections between \Le-diagrams, and various
combinatorial objects: decorated permutations, some matroids called positroids,
Grassmann necklaces. He also shows that the number of
\Le-diagram of shape $\lambda$ is equal to the number of acyclic
orientations of some graph $G_\lambda$. For a partition $\lambda\subset 
(n-k)^k$, the graph $G_\lambda$ is
the bipartite graph on the vertices $1\ldots k$ and
$1'\ldots(n-k)'$, with edges $(i,j')$ corresponding to cells
$(i,j)$ of the Young diagram defined by $\lambda$.

\medskip

The proof goes recursively. Let $f_\lambda(j)$ be
the number of \Le-diagrams of shape $\lambda$ filled with $j$ ones
and let $F_\lambda(q)$ be the polynomial $\sum_k f_\lambda(j)q^j$.
The polynomial $F_\lambda(q)$ satisfies a very simple recurrence
derived by L. Williams in \cite{LW} for a fixed corner and
generalized by A. Postnikov \cite{AP}. Indeed let us pick a corner
box $x$ of the Young diagram of shape $\lambda$. Let
$\lambda^{(1)}$, $\lambda^{(2)}$, $\lambda^{(3)}$ and
$\lambda^{(4)}$, be the Young diagrams obtained from $\lambda$ by
removing, respectively, the box $x$, the row containing $x$, the
column containing $x$, the column and the row containing $x$. Then
it is easy to see that $F_\lambda(q)=1$ if $|\lambda|=0$ and
\begin{equation}
F_\lambda(q)=qF_{\lambda^{(1)}}(q)+F_{\lambda^{(2)}}(q)
+F_{\lambda^{(3)}}(q)-F_{\lambda^{(4)}}(q) \label{F}
\end{equation}
otherwise.

\bigskip

Let $\chi_\lambda(t)$ be the chromatic polynomial of the graph
$G_\lambda$. A. Postnikov \cite{AP} establishes that
$\chi_\lambda(t)=1$ if $|\lambda|=0$ and
\begin{equation}
\chi_\lambda(t)=\chi_{\lambda^{(1)}}(t)-t^{-1}\big(\chi_{\lambda^{(2)}}
(t)+\chi_{\lambda^{(3)}}(t)-\chi_{\lambda^{(4)}}(t)\big) \label{chi}
\end{equation}
otherwise. According to \cite{Sta}, the value
$(-1)^n\chi_\lambda(-1)$  equals the number $ao_\lambda$ of acyclic
orientations of the graph $G_\lambda$. Specializing equation \eqref{F} 
at $q=1$ and \eqref{chi} at $t=-1$, one obtains
that $ao_\lambda$ and $F_\lambda(1)$ satisfy the same recurrence
and have the same boundary condition and are therefore equal for
any $\lambda$.

\bigskip

The acyclic orientations of the graph $G_\lambda$ are in bijection
with some fillings of the diagram of $\lambda$ with 0s and 1s
which we call {\it X-diagrams}. A diagram is said to be an {\it
X-diagram}, if it avoids the patterns \pattern 1001 and \pattern 0110.
This bijection is very simple: the filling of a cell $(i,j)$ is 0
(resp. 1) if and only if the orientation of the edge $(i,j')$ is
$i\rightarrow j'$ (resp. $i\leftarrow j'$). One can check
that the pattern-avoidance for the X-diagrams is equivalent to the
cycle avoidance for the orientation. Details can be found in
\cite{AP,AS}.

\bigskip

Therefore X-diagrams and \Le-diagrams are equivalent in the
following sense: \begin{prop}\cite{AP,AS} For every Young diagram
$\lambda$, the number of X-diagrams of shape $\lambda$ is equal to
the number of \Le-diagrams of shape $\lambda$.
\end{prop}

Postnikov \cite{AP} and Burstein \cite{Bu} noticed that \Le-diagrams 
are also equivalent to the diagrams avoiding \pattern 0111 and \pattern 1111.
A. Spiridonov \cite{AS} then made an extensive study on which
pairs of patterns are equivalent, for more general polyominoes than
Young diagrams. He proved that many other patterns follow
similar recurrence relations as in \cite{AP}, and he proved
the equivalence of X-diagrams and \Le-diagrams for a whole 
class of polyominoes, containing for example skew shapes (in French 
notation). 

\bigskip

In this article, we  give bijective proofs for the equivalence of
the two main families of diagrams:
\begin{itemize}
\item the \Le-diagrams which avoids the patterns \pattern 1110 and \pattern 0110,
\smallskip

\item and the X-diagrams which avoids the patterns \pattern 1001 and \pattern 0110.
\smallskip

\end{itemize}
This is done first in the case of Young diagram, but we extend the 
bijection to other polyominoes.

\begin{definition}
A polyomino $S$ is {\it \Le-complete} if satisfies the following 
conditions: 
for any $i<j$ and $k<l$, if $(j,k), (j,l), (i,k)$ are cells of $S$ then
the cell $(i,l)$ is also in $S$. Otherly said, for any three cells arranged
as a $\Le$, there is a $2\times 2$-submatrix of $S$ containing these three
cells.
\end{definition}

This is equivalent to recursive condition ``2$\times$2-connected bottom-right 
CR-erasable" of \cite{AS}. We give a bijection between X-diagrams and \Le-diagrams
for any \Le-complete polyomino. There are some parameters preserved by 
this bijection.

\begin{definition}
  We call {\it zero-row} (resp. {\it zero-column}) a row (resp. column) filled 
with 0s.
  A row of a diagram is restricted, if it contains a 0 having a 1 above it 
in the same column. Otherwise it is called unrestricted.
\end{definition}

The first bijection we describe preserves every zero-row and zero-column.
With a slightly different construction, we have a bijection preserving 
every zero-column and unrestricted row.

\bigskip

As the bijection conserves the zero-columns, we get a
direct corollary linking {\it permutation tableaux} \cite{Bu,CN,SW}
(which are \Le-diagrams with no zero column) and  X-diagrams with
no zero-column:
\begin{corollary}
There exists a bijection between X-diagrams of shape $\lambda$
with no zero-column and $k$ unrestricted rows and permutation
tableaux of shape $\lambda$ with $k$ unrestricted rows.
\end{corollary}

The {\it length} of a diagram is the number of columns
plus the number of rows of the diagrams. Permutation tableaux of
length $n$ with $k$ unrestricted rows (allowing rows of size 0)
are known to be enumerated by
the Stirling numbers of the first kind \cite{CN}. We immediately
get that:
\begin{corollary}
The number of X-diagrams of length $n$ with no zero column and $k$ 
unrestricted rows is equal to the number of permutations of 
$\{1,\ldots ,n\}$ with $k$ cycles.
\end{corollary}

\bigskip

This article is organized as follows. Section 2 contains 
elementary results that are direct consequences of the previous 
definitions. In Section 3 we define the main bijection of
this article, which answers the original problem given by Postnikov.
In Section 4, we extend this bijection to more general polyominoes
than Young diagrams, proving bijectively some results of Spiridonov.
In Section 5, we discuss various possible generalizations of these
results.

\bigskip

\section*{Acknowledgement}

I thank my advisor Sylvie Corteel for her precious help and guidance,
and Alexey Spiridonov for many interesting comments and suggestions.

\bigskip

\section{Firsts results on the structure of X-diagrams}

We give here several easy but useful lemmas to characterize X-diagrams.
Moreover this will give a raw outline of the methods we use throughout this
article.

\begin{lem}\label{Xstruct}
For any diagram $T$, the following conditions are equivalent:
\begin{enumerate}
\item $T$ is an X-diagram.
\item For every rectangular submatrix $M$ of $T$, two rows of $M$ having 
the same number of 1s are equal.
\item For any rectangular submatrix of $T$ having two rows, 
the index set of 1s in the first row contains, or is contained in, the index
set of 1s in the second row.
\item For every rectangular submatrix $M$ of $T$, the following condition
 holds: if an entry of $M$ contains a 0 and belongs to a row with a maximal
 number of 1s, then this entry belongs to a zero-column.
\end{enumerate}
\end{lem}

\begin{proof}
We show that the first condition implies the last one, all other implications
are proved with similar arguments. So let $M$ be a rectangular submatrix of 
an X-diagram (in particular $M$ is also an X-diagram). Let $x,y>0$ be such 
that $M(x,y)=0$ and the $x$th row of $M$ has a maximal number of 1s.
Now if there is $z$ such that $M(z,y)=1$, by the pattern-avoidance condition
we have $M(x,t)=1 \Longrightarrow M(z,t)=1$ for any $t\neq y$. This means that
the $z$th row has strictly more 1s than the $x$th, which contradicts the 
definition of $x$.
\end{proof}

\medskip

From these characterization, we obtain the following statement.

\begin{corollary} \label{X2}
Let $T$ be an X-diagram of rectangular shape.
\begin{itemize}
\item 
If $T'$ is obtained from $T$ by permuting rows, $T'$ is also an X-diagram.
\item If $T'$ is obtained from $T$ by replacing a row with a copy of another 
row, $T'$ is also an X-diagram.
\item
If $T$ and $T'$ have the same set of distinct rows, $T'$ is also an X-diagram.
\end{itemize}
\end{corollary}

There is also a similar statement with columns instead of rows. Notice that
a X-diagram of rectangular shape, after permuting its rows and its columns
can be arranged such that the 1s and the 0s are two complementary Young diagrams.

\bigskip

\section{The bijection $\Phi$ between \Le-diagrams and X-diagrams}

This bijection is defined recursively with respect to the number of rows.
We first define a bijection $\phi$ between X-diagrams and 
{\it mixed diagrams} (see definition below). 
Once $\phi$ is defined there is a short recursive
definition of $\Phi$.

\smallskip

Through this section we only consider diagrams whose shape is a Young diagram
in English notation.We use the convention that the top-left corner is the entry $(1,1)$, 
the bottom-left corner is the entry $(k,1)$, the top-right corner is the entry
$(1,\lambda_1)$, and so on. For any diagram $T$, we denote by $T(i,j)\in\{0,1\}$ 
the number in the cell $(i,j)$ of $T$.

\subsection{The bijection $\phi$ between X-diagrams and mixed diagrams}

Through this section, let $\lambda$ be a Young diagram, in English 
notation, and $k$ its number of rows.
The row lengths are a weakly decreasing sequence $\lambda_1 \geq
\dots \geq \lambda_k>0$.

\begin{definition} 
A {\it mixed diagram} of shape $\lambda$ is a diagram of 
shape $\lambda$ with the following properties:
\begin{itemize}
\item The $k-1$ top rows are an X-diagram,
\item For any 0 in the $k$th row, there is no 1 above it in the same column, 
or there is no 1 to its left in the same row.
\end{itemize}
\end{definition}

The goal of this section is to prove the following:

\begin{prop} \label{phi}
There exists a bijection $\phi$ between X-diagrams of shape $\lambda$ and mixed 
diagrams of shape $\lambda$. Moreover for any X-diagram $T$, 
$T$ and $\phi(T)$ have the same set of zero-columns  and zero-rows.
\end{prop}

\medskip

The essential tool for this bijection is given by the following definition.

\begin{definition} \label{pivot}
We call {\it pivot column} of an X-diagram of shape $\lambda$, a 
column among the $\lambda_k$ first ones ({\it i.e.} it is a column of
maximal size) such that:
\begin{itemize}
\item there is a 1 in bottom position,
\item it has a maximal number of 0s among columns satisfying the previous 
  property,
\item it is in leftmost position among columns satisfying the previous two
  properties.
\end{itemize}

If such a column exists, it is uniquely defined (by the third property).
There is no such column only in the case where the bottom row is a zero-row.
\end{definition}

\medskip

We now come to the description of $\phi$. Let $T$ be an X-diagram of shape
$\lambda$. In the case where there is no pivot column in $T$, otherly said the
bottom row is a zero-row, $T$ is also a mixed-diagram and we put $\phi(T)=T$.
Otherwise, we define $\phi(T)$ as the result of the following column-by-column
transformation of $T$, where $j$ is the index of its pivot column:

\medskip

\begin{itemize}
\item In a column of index $i$ with $i<j$, we change the bottom entry into a 0.
\item A column of index $i$ with $j<i\leq\lambda_k$, which is identical to the pivot
  column, is changed into a column having a 1 in bottom position and 0s
  elsewhere.
\item In a column of index $i$ with $j<i\leq\lambda_k$, which is not identical to the
 pivot column and is not a zero-column, we change the bottom entry into a 1.
\end{itemize}

\bigskip

Notice that this transformation only modifies the $\lambda_k$ leftmost columns
of $T$. For example, an X-diagram $T$ and its image $\phi(T)$ are given by:

\begin{figure}[h!tp]
\centering
\psset{unit=3.6mm}
\begin{pspicture}(-1.6,0)(12,5)
\psline(0,0)(9,0)\psline(0,1)(11,1)\psline(0,2)(12,2)
\psline(0,3)(12,3)\psline(0,4)(12,4)\psline(0,5)(12,5)
\psline(0,0)(0,5)
\psline(1,0)(1,5)\psline(2,0)(2,5)\psline(3,0)(3,5)\psline(4,0)(4,5)
\psline(5,0)(5,5)\psline(6,0)(6,5)\psline(7,0)(7,5)\psline(8,0)(8,5)
\psline(9,0)(9,5)\psline(10,1)(10,5)\psline(11,1)(11,5)\psline(12,2)(12,5)

\rput(3.5,0.5){\bf 1}\rput(3.5,1.5){\bf 0}\rput(3.5,2.5){\bf 0}
\rput(3.5,3.5){\bf 1}\rput(3.5,4.5){\bf 0}

\rput(0.5,0.5){1}\rput(0.5,1.5){1}\rput(0.5,2.5){1}\rput(0.5,3.5){1}\rput(0.5,4.5){1}
\rput(1.5,0.5){1}\rput(1.5,1.5){0}\rput(1.5,2.5){1}\rput(1.5,3.5){1}\rput(1.5,4.5){0}
\rput(2.5,0.5){0}\rput(2.5,1.5){0}\rput(2.5,2.5){0}\rput(2.5,3.5){1}\rput(2.5,4.5){0}

\rput(4.5,0.5){1}\rput(4.5,1.5){0}\rput(4.5,2.5){1}\rput(4.5,3.5){1}\rput(4.5,4.5){1}
\rput(5.5,0.5){0}\rput(5.5,1.5){0}\rput(5.5,2.5){0}\rput(5.5,3.5){1}\rput(5.5,4.5){0}
\rput(6.5,0.5){1}\rput(6.5,1.5){0}\rput(6.5,2.5){1}\rput(6.5,3.5){1}\rput(6.5,4.5){0}
\rput(7.5,0.5){0}\rput(7.5,1.5){0}\rput(7.5,2.5){0}\rput(7.5,3.5){0}\rput(7.5,4.5){0}
\rput(8.5,0.5){1}\rput(8.5,1.5){0}\rput(8.5,2.5){0}\rput(8.5,3.5){1}\rput(8.5,4.5){0}
                 \rput(9.5,1.5){0}\rput(9.5,2.5){1}\rput(9.5,3.5){1}\rput(9.5,4.5){0}
                 \rput(10.5,1.5){0}\rput(10.5,2.5){0}\rput(10.5,3.5){1}\rput(10.5,4.5){0}
                                  \rput(11.5,2.5){0}\rput(11.5,3.5){0}\rput(11.5,4.5){0}
\rput(-1.6,2.5){$T=$}
\end{pspicture},
\hspace{8mm}
\begin{pspicture}(-2.1,0)(12,5)
\psline(0,0)(9,0)\psline(0,1)(11,1)\psline(0,2)(12,2)
\psline(0,3)(12,3)\psline(0,4)(12,4)\psline(0,5)(12,5)
\psline(0,0)(0,5)
\psline(1,0)(1,5)\psline(2,0)(2,5)\psline(3,0)(3,5)\psline(4,0)(4,5)
\psline(5,0)(5,5)\psline(6,0)(6,5)\psline(7,0)(7,5)\psline(8,0)(8,5)
\psline(9,0)(9,5)\psline(10,1)(10,5)\psline(11,1)(11,5)\psline(12,2)(12,5)

\rput(3.5,0.5){\bf 1}\rput(3.5,1.5){\bf 0}\rput(3.5,2.5){\bf 0}
\rput(3.5,3.5){\bf 1}\rput(3.5,4.5){\bf 0}

\rput(0.5,0.5){0}\rput(0.5,1.5){1}\rput(0.5,2.5){1}\rput(0.5,3.5){1}\rput(0.5,4.5){1}
\rput(1.5,0.5){0}\rput(1.5,1.5){0}\rput(1.5,2.5){1}\rput(1.5,3.5){1}\rput(1.5,4.5){0}
\rput(2.5,0.5){0}\rput(2.5,1.5){0}\rput(2.5,2.5){0}\rput(2.5,3.5){1}\rput(2.5,4.5){0}

\rput(4.5,0.5){1}\rput(4.5,1.5){0}\rput(4.5,2.5){1}\rput(4.5,3.5){1}\rput(4.5,4.5){1}
\rput(5.5,0.5){1}\rput(5.5,1.5){0}\rput(5.5,2.5){0}\rput(5.5,3.5){1}\rput(5.5,4.5){0}
\rput(6.5,0.5){1}\rput(6.5,1.5){0}\rput(6.5,2.5){1}\rput(6.5,3.5){1}\rput(6.5,4.5){0}
\rput(7.5,0.5){0}\rput(7.5,1.5){0}\rput(7.5,2.5){0}\rput(7.5,3.5){0}\rput(7.5,4.5){0}
\rput(8.5,0.5){1}\rput(8.5,1.5){0}\rput(8.5,2.5){0}\rput(8.5,3.5){0}\rput(8.5,4.5){0}
                 \rput(9.5,1.5){0}\rput(9.5,2.5){1}\rput(9.5,3.5){1}\rput(9.5,4.5){0}
                 \rput(10.5,1.5){0}\rput(10.5,2.5){0}\rput(10.5,3.5){1}\rput(10.5,4.5){0}
                                  \rput(11.5,2.5){0}\rput(11.5,3.5){0}\rput(11.5,4.5){0}
\rput(-2.1,2.5){$\phi(T)=$}
\end{pspicture}.\end{figure}
Here the pivot column is the 4th one, with bolded numbers. We can see that the 9th 
column is identical to the pivot column. So it is replaced with a column having a 
single 1 in bottom position.

\bigskip

Now that the map is described, we prove that $\phi$ is indeed a bijection.
We begin with two lemmas that will be helpful to define
$\phi^{-1}$.

\bigskip

\begin{lem} \label{lem2}
 Let $U=\phi(T)$, and $j>0$. Then the following conditions are equivalent:
 \begin{itemize}
 \item The pivot column of $T$ is the $j$th column,
 \item We have $U(k,j)=1$, and $U(k,i)=0$ for any $i<j$.
 \end{itemize}
\end{lem}

\begin{proof} Each statement is either not satisfied by any $j$, or
satisfied by a unique $j$. So we just have to prove the direct implication.
If the pivot column of $T$ is the $j$th column, it is not modified so we have
$U(k,j)=1$. And for any $i<j$, to obtain $U$ we put a 0 in the entry $(k,i)$ of $T$.
So if the first condition is true, the second too, hence the equivalence.
\end{proof}

\smallskip

\begin{lem} \label{lem3}
  Let $j$ be the index of the pivot column of $T$. Then for any $i<j$ the following 
  conditions are equivalent:
  \begin{itemize}
  \item The $k-1$ top entries of the $i$th column contains strictly less
    0s than the $k-1$ top entries of the $j$th column,
  \item $T(k,i)=1$.
  \end{itemize}
\end{lem}

\begin{proof} If the $k-1$ top entries of the $i$th column contains 
  strictly more 0s than the $k-1$ top entries of the $j$th column, we have
  $T(k,i)=0$ since otherwise it would contradict the second point in the 
  definition of the pivot column (maximum number of 0s).
  
  If the $k-1$ top entries of the $i$th column contains exactly as many 0s 
  as the $k-1$ top entries of the $j$th column, then $T(k,i)=0$ since 
  otherwise it would contradict the third point in the definition of the pivot 
  column (recall that two columns having the same number number of 1s are 
  identical in an X-diagram).
  
  Eventually the last case to check is when the $k-1$ top entries of the $i$th 
  column contains strictly less 0s than the $k-1$ top entries of the $j$th 
  column. Then $T(k,i)=1$ since otherwise there would be an occurrence of the 
  pattern \pattern 1001.
\end{proof}

\medskip

The same method also gives a proof of:

\begin{lem} \label{lem4}
  Let $j$ be the index of the pivot column of $T$. Then for any 
  $i$ such that $j<i\leq\lambda_k$ we have:
  \begin{itemize}
  \item If the $k-1$ top entries of the $i$th column contains strictly more
    1s than the $k-1$ top entries of the $j$th column, then $T(k,i)=1$.
  \item If the $k-1$ top entries of the $i$th column contains strictly less
    1s than the $k-1$ top entries of the $j$th column, then $T(k,i)=0$.
  \end{itemize}
\end{lem}

\noindent
{\it Remark.} In the previous lemma there is nothing about the case of equality.
Indeed, it is possible that some columns are identical to the pivot column, but 
some other columns are identical except that the  bottom entry is different.
To prove that $\phi$ is a bijection it is important to notice that in any case,
two differently-filled columns of $T$ give rise to
differently-filled columns in $\phi(T)$. An essential
remark is that if there is in $T$ a column with a single 1 in bottom position 
and 0s elsewhere, then it is necessarly identical to the pivot column, so that 
there is no confusion possible.

\medskip

With these lemmas we have all information to define the inverse map of $\phi$.
Let $U$ be a mixed diagram. If the bottom row is a zero-row, $U$ is also a
X-diagram and we have $\phi(U)=U$. Otherwise, we can find an X-diagram $T$ 
such that $\phi(T)=U$. Indeed, $T$ can be defined from $U$ with the 
following column-by-column transformation, where $j$ is the index of
the leftmost column in $U$ having a 1 in bottom position:

\medskip

\begin{itemize}
\item In a column of index $i$ such that $i<j$, and where
  the $k-1$ top entries contains strictly less 0s than the $j$th 
  column, the bottom entry is replaced with a 1.

\item A column of index $i$ such that $j<i\leq\lambda_k$, which has
 a 1 in bottom position and 0s elsewhere,
 is replaced with a copy of the $j$th column.

\item Eventually, consider a column of index $i$ such that 
  $j<i\leq\lambda_k$, which does not satisfy the previous condition.
  If the $k-1$ top entries of this column contains at most as many
  1s as the $j$th column, then the last entry is replaced with a 0.
\end{itemize}

\medskip

This transformation is the step-by-step inverse process of the 
transformation
defining $\phi$. So when composing the two maps, with $\phi$ in first 
(resp. in last), we get the identity on X-diagrams (resp.
mixed diagrams). So we are able to prove Proposition \ref{phi}.

\begin{proof} It remains only to check that $T$ and $U=\phi(T)$ have the same 
  zero-columns and zero-rows. The fact about columns is clear since
  the algorithm is a column-by-column transformation. Now let $i<k$ be the
  index of a non-zero row.
  The $i$th row of $T$ may be modified by the algorithm only when 
  $T(i,j)=1$ where $j$ is the index of the pivot column (when transforming 
  a column equal to the pivot column to a column having a single 1). But 
  then we have also $U(i,j)=1$, so the $i$th row of $U$ is non-zero.
  Hence $T$ and $U$ have the same set of zero-rows.
\end{proof}

\smallskip

\subsection{The bijection $\Phi$ between X-diagrams and \Le-diagrams}

Now we describe the main bijection of this article.
Given an X-diagram $T$ of shape $\lambda$,
$\Phi(T)$ is defined by transforming $T$ with the following 
algorithm:
\smallskip

\noindent
\begin{tabular}{|l|}
 \hline
  For $i$ from $k$ to $1$, {\huge\phantom i}
  replace the $i$ top rows of $T$ with their image by $\phi$.\\
 \hline
\end{tabular}
\smallskip

\noindent
Equivalently, we obtain $\Phi(T)$ by applying recursively $\Phi$ to the $k-1$ first
rows of $\phi(T)$.

\medskip

\begin{lem} If the $i$th column of $T$ is a zero-column, then the $i$th 
column of $\Phi(T)$ is also a zero-column.
\end{lem}

\begin{proof}
It is a direct consequence of the fact that $\phi$ preserves
every zero-column.
\end{proof}

\smallskip

\begin{lem}\label{Le}
 The diagram $\Phi(T)$ is a \Le-diagram.
\end{lem}

\begin{proof} 
We obtain $\Phi(T)$ by applying $\Phi$ to the $k-1$ first 
rows of $\phi(T)$, so we can prove the result recursively on the number
of rows. For $k=1$, every diagram with just one row is an X-diagram and a 
\Le-diagram, and $\Phi$ is the identity in this case.

\medskip

For $k>1$, under the recurrence assumption a 0 in the $k-1$ top rows of $\Phi(T)$
cannot have a 1 to its left and a 1 above it. Now suppose there is a 0 in 
position $(k,i)$, ({\it i.e.} in the $k$th row) having a 1 to its left. We have to
prove that the $i$th column of $\Phi(T)$ is a zero-column.

\medskip

Since the $k$th row of $\Phi(T)$ is the same as in $\phi(T)$, we have
$\phi(T)(k,i)=0$ and this 0 has a 1 to its left. But since $\phi(T)$ is a mixed
diagram, it implies that the $i$th column of $\phi(T)$ is a zero-column.
And since $\Phi$ preserves every zero-column, it implies that the 
$i$th column of $\Phi(T)$ is also a zero-column. This completes the proof.
\end{proof}

\smallskip

\begin{prop} \label{pphi}
The map $\Phi$ is a bijection between X-diagrams of
shape $\lambda$ and \Le-diagrams of shape $\lambda$. Moreover this bijection
preserves the set of zero-rows and the set of zero-columns.
\end{prop}

\begin{proof}
At every step, the upper part of $T$ still avoids the patterns
\pattern 1001 and \pattern 0110 , so that it makes sense to take its image 
by $\phi$. Now, for every \Le-diagram U of shape $\lambda$, define 
$\Phi^{-1}(U)$ by transforming $U$ with the following algorithm:
\smallskip

\noindent
\begin{tabular}{|l|}
 \hline
  For $i$ from $1$ to $k$, {\huge\phantom i}
  replace the $i$ top rows of $U$ with their image by $\phi^{-1}$.\\
 \hline
\end{tabular}
\smallskip

\noindent
Equivalently, we can replace the $k-1$ top rows of $U$ with their image by
$\Phi^{-1}$ and then apply $\phi^{-1}$ to obtain $\Phi^{-1}(U)$.
With Lemma \ref{Le} and knowing that $\phi$ is bijective, it is clear that 
$\Phi^{-1}\circ\Phi$ is the identity on X-diagrams and $\Phi\circ\Phi^{-1}$ is
the  identity on \Le-diagrams, so that $\Phi$ is bijective and $\Phi^{-1}$ 
as we defined it is indeed its inverse. Moreover, the fact 
that $\phi$ and $\phi^{-1}$ preserve every zero-column, directly 
implies the same fact for $\Phi$ and $\Phi^{-1}$. Similarly, the fact that 
$\phi$ preserves the set of zero-rows implies the same property for 
$\Phi$.
\end{proof}

\subsection{Examples}

This example contains no zero-column or zero-row, since they would be
unchanged at any step of the process. We start from an X-diagram $T$, and 
compute step by step its image by $\Phi$. So each step gives an example of 
a diagram and its image by $\phi$. The thick lines 
indicate where is the upper-left part, which is to be replaced with its 
image by $\phi$. Bold numbers indicate the pivot column. Suppose that we have:

\[
\mbox{
\psset{unit=0.4cm}
\begin{pspicture}(0,0)(2,5)
\rput(0.5,2.5){$T=$}
\end{pspicture}
}
\mbox{
\psset{unit=0.4cm}
\begin{pspicture}(0,0)(7,5)
  \psframe[dimen=middle](0,4)(6,5)
  \psframe[dimen=middle](0,3)(6,4)
  \psframe[dimen=middle](0,2)(5,3)
  \psframe[dimen=middle](0,1)(4,2)
  \psframe[dimen=middle](0,0)(3,1)

  \psframe[dimen=middle](0,0)(1,5)
  \psframe[dimen=middle](1,0)(2,5)
  \psframe[dimen=middle](2,0)(3,5)
  \psframe[dimen=middle](3,1)(4,5)
  \psframe[dimen=middle](4,2)(5,5)
  \psframe[dimen=middle](5,3)(6,5)

  \rput(0.5,4.5){\small 1}
  \rput(1.5,4.5){\small 1}
  \rput(2.5,4.5){\small 1}
  \rput(3.5,4.5){\small 1}
  \rput(4.5,4.5){\small 1}
  \rput(5.5,4.5){\small 1}

  \rput(0.5,3.5){\small 0}
  \rput(1.5,3.5){\small 0}
  \rput(2.5,3.5){\small 0}
  \rput(3.5,3.5){\small 1}
  \rput(4.5,3.5){\small 0}
  \rput(5.5,3.5){\small 0}

  \rput(0.5,2.5){\small 1}
  \rput(1.5,2.5){\small 1}
  \rput(2.5,2.5){\small 0}
  \rput(3.5,2.5){\small 1}
  \rput(4.5,2.5){\small 0}

  \rput(0.5,1.5){\small 1}
  \rput(1.5,1.5){\small 0}
  \rput(2.5,1.5){\small 0}
  \rput(3.5,1.5){\small 1}

  \rput(0.5,0.5){\small 1}
  \rput(1.5,0.5){\small 1}
  \rput(2.5,0.5){\small 1}
\end{pspicture}}.\]
\medskip

\noindent
In the first step, the index of the pivot column is $j=3$. We have to put 
0s in the $j-1$ leftmost entries of the bottom row. Since there is no column to
the right of the pivot column, there is nothing else to do. So the 
transformation is:
\[
\mbox{
\psset{unit=0.4cm}
\begin{pspicture}(0,0)(7,5)
  \psframe[linewidth=0.6mm](0,0)(3,5)
  \psframe[dimen=middle](0,4)(6,5)
  \psframe[dimen=middle](0,3)(6,4)
  \psframe[dimen=middle](0,2)(5,3)
  \psframe[dimen=middle](0,1)(4,2)
  \psframe[dimen=middle](0,0)(3,1)

  \psframe[dimen=middle](0,0)(1,5)
  \psframe[dimen=middle](1,0)(2,5)
  \psframe[dimen=middle](2,0)(3,5)
  \psframe[dimen=middle](3,1)(4,5)
  \psframe[dimen=middle](4,2)(5,5)
  \psframe[dimen=middle](5,3)(6,5)

  \rput(0.5,4.5){\small 1}
  \rput(1.5,4.5){\small 1}
  \rput(2.5,4.5){\small\bf 1}
  \rput(3.5,4.5){\small 1}
  \rput(4.5,4.5){\small 1}
  \rput(5.5,4.5){\small 1}

  \rput(0.5,3.5){\small 0}
  \rput(1.5,3.5){\small 0}
  \rput(2.5,3.5){\small\bf 0}
  \rput(3.5,3.5){\small 1}
  \rput(4.5,3.5){\small 0}
  \rput(5.5,3.5){\small 0}

  \rput(0.5,2.5){\small 1}
  \rput(1.5,2.5){\small 1}
  \rput(2.5,2.5){\small\bf 0}
  \rput(3.5,2.5){\small 1}
  \rput(4.5,2.5){\small 0}

  \rput(0.5,1.5){\small 1}
  \rput(1.5,1.5){\small 0}
  \rput(2.5,1.5){\small\bf 0}
  \rput(3.5,1.5){\small 1}

  \rput(0.5,0.5){\small 1}
  \rput(1.5,0.5){\small 1}
  \rput(2.5,0.5){\small\bf 1}
\end{pspicture}}
\mbox{
\psset{unit=0.4cm}
\begin{pspicture}(0,0)(7,5)
\rput(3.5,2.5){$\longrightarrow$}
\end{pspicture}
}
\mbox{
\psset{unit=0.4cm}
\begin{pspicture}(0,0)(7,5)
  \psframe[linewidth=0.6mm](0,0)(3,5)
  \psframe[dimen=middle](0,4)(6,5)
  \psframe[dimen=middle](0,3)(6,4)
  \psframe[dimen=middle](0,2)(5,3)
  \psframe[dimen=middle](0,1)(4,2)
  \psframe[dimen=middle](0,0)(3,1)

  \psframe[dimen=middle](0,0)(1,5)
  \psframe[dimen=middle](1,0)(2,5)
  \psframe[dimen=middle](2,0)(3,5)
  \psframe[dimen=middle](3,1)(4,5)
  \psframe[dimen=middle](4,2)(5,5)
  \psframe[dimen=middle](5,3)(6,5)

  \rput(0.5,4.5){\small 1}
  \rput(1.5,4.5){\small 1}
  \rput(2.5,4.5){\small 1}
  \rput(3.5,4.5){\small 1}
  \rput(4.5,4.5){\small 1}
  \rput(5.5,4.5){\small 1}

  \rput(0.5,3.5){\small 0}
  \rput(1.5,3.5){\small 0}
  \rput(2.5,3.5){\small 0}
  \rput(3.5,3.5){\small 1}
  \rput(4.5,3.5){\small 0}
  \rput(5.5,3.5){\small 0}

  \rput(0.5,2.5){\small 1}
  \rput(1.5,2.5){\small 1}
  \rput(2.5,2.5){\small 0}
  \rput(3.5,2.5){\small 1}
  \rput(4.5,2.5){\small 0}

  \rput(0.5,1.5){\small 1}
  \rput(1.5,1.5){\small 0}
  \rput(2.5,1.5){\small 0}
  \rput(3.5,1.5){\small 1}

  \rput(0.5,0.5){\small 0}
  \rput(1.5,0.5){\small 0}
  \rput(2.5,0.5){\small 1}
\end{pspicture}}\]

\noindent
In the second step, the index of the pivot column is $j=1$. 
So the transformation is:
\[
\mbox{
\psset{unit=0.4cm}
\begin{pspicture}(0,0)(7,5)
  \psframe[linewidth=0.7mm](0,1)(4,5)
  \psframe[dimen=middle](0,4)(6,5)
  \psframe[dimen=middle](0,3)(6,4)
  \psframe[dimen=middle](0,2)(5,3)
  \psframe[dimen=middle](0,1)(4,2)
  \psframe[dimen=middle](0,0)(3,1)

  \psframe[dimen=middle](0,0)(1,5)
  \psframe[dimen=middle](1,0)(2,5)
  \psframe[dimen=middle](2,0)(3,5)
  \psframe[dimen=middle](3,1)(4,5)
  \psframe[dimen=middle](4,2)(5,5)
  \psframe[dimen=middle](5,3)(6,5)

  \rput(0.5,4.5){\small\bf 1}
  \rput(1.5,4.5){\small 1}
  \rput(2.5,4.5){\small 1}
  \rput(3.5,4.5){\small 1}
  \rput(4.5,4.5){\small 1}
  \rput(5.5,4.5){\small 1}

  \rput(0.5,3.5){\small\bf 0}
  \rput(1.5,3.5){\small 0}
  \rput(2.5,3.5){\small 0}
  \rput(3.5,3.5){\small 1}
  \rput(4.5,3.5){\small 0}
  \rput(5.5,3.5){\small 0}

  \rput(0.5,2.5){\small\bf 1}
  \rput(1.5,2.5){\small 1}
  \rput(2.5,2.5){\small 0}
  \rput(3.5,2.5){\small 1}
  \rput(4.5,2.5){\small 0}

  \rput(0.5,1.5){\small\bf 1}
  \rput(1.5,1.5){\small 0}
  \rput(2.5,1.5){\small 0}
  \rput(3.5,1.5){\small 1}

  \rput(0.5,0.5){\small 0}
  \rput(1.5,0.5){\small 0}
  \rput(2.5,0.5){\small 1}
\end{pspicture}} 
\mbox{
\psset{unit=0.4cm}
\begin{pspicture}(0,0)(7,5)
\rput(3.5,2.5){$\longrightarrow$}
\end{pspicture}
}
\mbox{
\psset{unit=0.4cm}
\begin{pspicture}(0,0)(7,5)
  \psframe[linewidth=0.7mm](0,1)(4,5)
  \psframe[dimen=middle](0,4)(6,5)
  \psframe[dimen=middle](0,3)(6,4)
  \psframe[dimen=middle](0,2)(5,3)
  \psframe[dimen=middle](0,1)(4,2)
  \psframe[dimen=middle](0,0)(3,1)

  \psframe[dimen=middle](0,0)(1,5)
  \psframe[dimen=middle](1,0)(2,5)
  \psframe[dimen=middle](2,0)(3,5)
  \psframe[dimen=middle](3,1)(4,5)
  \psframe[dimen=middle](4,2)(5,5)
  \psframe[dimen=middle](5,3)(6,5)

  \rput(0.5,4.5){\small 1}
  \rput(1.5,4.5){\small 1}
  \rput(2.5,4.5){\small 1}
  \rput(3.5,4.5){\small 1}
  \rput(4.5,4.5){\small 1}
  \rput(5.5,4.5){\small 1}

  \rput(0.5,3.5){\small 0}
  \rput(1.5,3.5){\small 0}
  \rput(2.5,3.5){\small 0}
  \rput(3.5,3.5){\small 1}
  \rput(4.5,3.5){\small 0}
  \rput(5.5,3.5){\small 0}

  \rput(0.5,2.5){\small 1}
  \rput(1.5,2.5){\small 1}
  \rput(2.5,2.5){\small 0}
  \rput(3.5,2.5){\small 1}
  \rput(4.5,2.5){\small 0}

  \rput(0.5,1.5){\small 1}
  \rput(1.5,1.5){\small 1}
  \rput(2.5,1.5){\small 1}
  \rput(3.5,1.5){\small 1}

  \rput(0.5,0.5){\small 0}
  \rput(1.5,0.5){\small 0}
  \rput(2.5,0.5){\small 1}
\end{pspicture}
}
\]

\noindent
In the third step, we have $j=1$. Here the second column is identical to 
the pivot column so the transformation is:
\[
\mbox{
\psset{unit=0.4cm}
\begin{pspicture}(0,0)(7,5)
  \psframe[linewidth=0.7mm](0,2)(5,5)
  \psframe[dimen=middle](0,4)(6,5)
  \psframe[dimen=middle](0,3)(6,4)
  \psframe[dimen=middle](0,2)(5,3)
  \psframe[dimen=middle](0,1)(4,2)
  \psframe[dimen=middle](0,0)(3,1)

  \psframe[dimen=middle](0,0)(1,5)
  \psframe[dimen=middle](1,0)(2,5)
  \psframe[dimen=middle](2,0)(3,5)
  \psframe[dimen=middle](3,1)(4,5)
  \psframe[dimen=middle](4,2)(5,5)
  \psframe[dimen=middle](5,3)(6,5)

  \rput(0.5,4.5){\small\bf 1}
  \rput(1.5,4.5){\small 1}
  \rput(2.5,4.5){\small 1}
  \rput(3.5,4.5){\small 1}
  \rput(4.5,4.5){\small 1}
  \rput(5.5,4.5){\small 1}

  \rput(0.5,3.5){\small\bf 0}
  \rput(1.5,3.5){\small 0}
  \rput(2.5,3.5){\small 0}
  \rput(3.5,3.5){\small 1}
  \rput(4.5,3.5){\small 0}
  \rput(5.5,3.5){\small 0}

  \rput(0.5,2.5){\small\bf 1}
  \rput(1.5,2.5){\small 1}
  \rput(2.5,2.5){\small 0}
  \rput(3.5,2.5){\small 1}
  \rput(4.5,2.5){\small 0}

  \rput(0.5,1.5){\small 1}
  \rput(1.5,1.5){\small 1}
  \rput(2.5,1.5){\small 1}
  \rput(3.5,1.5){\small 1}

  \rput(0.5,0.5){\small 0}
  \rput(1.5,0.5){\small 0}
  \rput(2.5,0.5){\small 1}
\end{pspicture}}
\mbox{
\psset{unit=0.4cm}
\begin{pspicture}(0,0)(7,5)
\rput(3.5,2.5){$\longrightarrow$}
\end{pspicture}
}
\mbox{
\psset{unit=0.4cm}
\begin{pspicture}(0,0)(7,5)
  \psframe[linewidth=0.7mm](0,2)(5,5)
  \psframe[dimen=middle](0,4)(6,5)
  \psframe[dimen=middle](0,3)(6,4)
  \psframe[dimen=middle](0,2)(5,3)
  \psframe[dimen=middle](0,1)(4,2)
  \psframe[dimen=middle](0,0)(3,1)

  \psframe[dimen=middle](0,0)(1,5)
  \psframe[dimen=middle](1,0)(2,5)
  \psframe[dimen=middle](2,0)(3,5)
  \psframe[dimen=middle](3,1)(4,5)
  \psframe[dimen=middle](4,2)(5,5)
  \psframe[dimen=middle](5,3)(6,5)

  \rput(0.5,4.5){\small 1}
  \rput(1.5,4.5){\small 0}
  \rput(2.5,4.5){\small 1}
  \rput(3.5,4.5){\small 1}
  \rput(4.5,4.5){\small 1}
  \rput(5.5,4.5){\small 1}

  \rput(0.5,3.5){\small 0}
  \rput(1.5,3.5){\small 0}
  \rput(2.5,3.5){\small 0}
  \rput(3.5,3.5){\small 1}
  \rput(4.5,3.5){\small 0}
  \rput(5.5,3.5){\small 0}

  \rput(0.5,2.5){\small 1}
  \rput(1.5,2.5){\small 1}
  \rput(2.5,2.5){\small 1}
  \rput(3.5,2.5){\small 1}
  \rput(4.5,2.5){\small 1}

  \rput(0.5,1.5){\small 1}
  \rput(1.5,1.5){\small 1}
  \rput(2.5,1.5){\small 1}
  \rput(3.5,1.5){\small 1}

  \rput(0.5,0.5){\small 0}
  \rput(1.5,0.5){\small 0}
  \rput(2.5,0.5){\small 1}
\end{pspicture}}\]

\noindent
In the fourth  step, the index of the pivot column is $j=4$.
\[
\mbox{
\psset{unit=0.4cm}
\begin{pspicture}(0,0)(7,5)
  \psframe[linewidth=0.7mm](0,3)(6,5)
  \psframe[dimen=middle](0,4)(6,5)
  \psframe[dimen=middle](0,3)(6,4)
  \psframe[dimen=middle](0,2)(5,3)
  \psframe[dimen=middle](0,1)(4,2)
  \psframe[dimen=middle](0,0)(3,1)

  \psframe[dimen=middle](0,0)(1,5)
  \psframe[dimen=middle](1,0)(2,5)
  \psframe[dimen=middle](2,0)(3,5)
  \psframe[dimen=middle](3,1)(4,5)
  \psframe[dimen=middle](4,2)(5,5)
  \psframe[dimen=middle](5,3)(6,5)

  \rput(0.5,4.5){\small 1}
  \rput(1.5,4.5){\small 0}
  \rput(2.5,4.5){\small 1}
  \rput(3.5,4.5){\small\bf 1}
  \rput(4.5,4.5){\small 1}
  \rput(5.5,4.5){\small 1}

  \rput(0.5,3.5){\small 0}
  \rput(1.5,3.5){\small 0}
  \rput(2.5,3.5){\small 0}
  \rput(3.5,3.5){\small\bf 1}
  \rput(4.5,3.5){\small 0}
  \rput(5.5,3.5){\small 0}

  \rput(0.5,2.5){\small 1}
  \rput(1.5,2.5){\small 1}
  \rput(2.5,2.5){\small 1}
  \rput(3.5,2.5){\small 1}
  \rput(4.5,2.5){\small 1}

  \rput(0.5,1.5){\small 1}
  \rput(1.5,1.5){\small 1}
  \rput(2.5,1.5){\small 1}
  \rput(3.5,1.5){\small 1}

  \rput(0.5,0.5){\small 0}
  \rput(1.5,0.5){\small 0}
  \rput(2.5,0.5){\small 1}
\end{pspicture}}
\mbox{
\psset{unit=0.4cm}
\begin{pspicture}(0,0)(7,5)
\rput(3.5,2.5){$\longrightarrow$}
\end{pspicture}
}
\mbox{
\psset{unit=0.4cm}
\begin{pspicture}(0,0)(7,5)
  \psframe[linewidth=0.7mm](0,3)(6,5)
  \psframe[dimen=middle](0,4)(6,5)
  \psframe[dimen=middle](0,3)(6,4)
  \psframe[dimen=middle](0,2)(5,3)
  \psframe[dimen=middle](0,1)(4,2)
  \psframe[dimen=middle](0,0)(3,1)

  \psframe[dimen=middle](0,0)(1,5)
  \psframe[dimen=middle](1,0)(2,5)
  \psframe[dimen=middle](2,0)(3,5)
  \psframe[dimen=middle](3,1)(4,5)
  \psframe[dimen=middle](4,2)(5,5)
  \psframe[dimen=middle](5,3)(6,5)

  \rput(0.5,4.5){\small 1}
  \rput(1.5,4.5){\small 0}
  \rput(2.5,4.5){\small 1}
  \rput(3.5,4.5){\small 1}
  \rput(4.5,4.5){\small 1}
  \rput(5.5,4.5){\small 1}

  \rput(0.5,3.5){\small 0}
  \rput(1.5,3.5){\small 0}
  \rput(2.5,3.5){\small 0}
  \rput(3.5,3.5){\small 1}
  \rput(4.5,3.5){\small 1}
  \rput(5.5,3.5){\small 1}

  \rput(0.5,2.5){\small 1}
  \rput(1.5,2.5){\small 1}
  \rput(2.5,2.5){\small 1}
  \rput(3.5,2.5){\small 1}
  \rput(4.5,2.5){\small 1}

  \rput(0.5,1.5){\small 1}
  \rput(1.5,1.5){\small 1}
  \rput(2.5,1.5){\small 1}
  \rput(3.5,1.5){\small 1}

  \rput(0.5,0.5){\small 0}
  \rput(1.5,0.5){\small 0}
  \rput(2.5,0.5){\small 1}
\end{pspicture}}\]

\noindent
So finally, we have:
\[
\mbox{
\psset{unit=0.4cm}
\begin{pspicture}(0,0)(2,5)
\rput(0.5,2.5){$\Phi(T)=$}
\end{pspicture}
}
\mbox{
\psset{unit=0.4cm}
\begin{pspicture}(0,0)(7,5)
  \psframe[dimen=middle](0,4)(6,5)
  \psframe[dimen=middle](0,3)(6,4)
  \psframe[dimen=middle](0,2)(5,3)
  \psframe[dimen=middle](0,1)(4,2)
  \psframe[dimen=middle](0,0)(3,1)

  \psframe[dimen=middle](0,0)(1,5)
  \psframe[dimen=middle](1,0)(2,5)
  \psframe[dimen=middle](2,0)(3,5)
  \psframe[dimen=middle](3,1)(4,5)
  \psframe[dimen=middle](4,2)(5,5)
  \psframe[dimen=middle](5,3)(6,5)
  \rput(0.5,4.5){\small 1}
  \rput(1.5,4.5){\small 0}
  \rput(2.5,4.5){\small 1}
  \rput(3.5,4.5){\small 1}
  \rput(4.5,4.5){\small 1}
  \rput(5.5,4.5){\small 1}

  \rput(0.5,3.5){\small 0}
  \rput(1.5,3.5){\small 0}
  \rput(2.5,3.5){\small 0}
  \rput(3.5,3.5){\small 1}
  \rput(4.5,3.5){\small 1}
  \rput(5.5,3.5){\small 1}

  \rput(0.5,2.5){\small 1}
  \rput(1.5,2.5){\small 1}
  \rput(2.5,2.5){\small 1}
  \rput(3.5,2.5){\small 1}
  \rput(4.5,2.5){\small 1}

  \rput(0.5,1.5){\small 1}
  \rput(1.5,1.5){\small 1}
  \rput(2.5,1.5){\small 1}
  \rput(3.5,1.5){\small 1}

  \rput(0.5,0.5){\small 0}
  \rput(1.5,0.5){\small 0}
  \rput(2.5,0.5){\small 1}
\end{pspicture}}\]
which is easily checked to be a \Le-diagram.

\bigskip

\section{Generalization to polyominoes other than Young diagrams}

Now, besides Young diagrams in English notation, we consider more general
polyominoes. In this Section we show that 
X-diagrams and \Le-diagrams are equivalent for any {\it \Le-complete} 
polyomino.
(see Definition \ref{lecomplete}).

\medskip

First we generalize the bijection $\phi$ for diagrams whose shape is a Young
diagram in French notation. Then we derive the generalization of the bijection
$\Phi$ for any \Le-complete polyomino, and we end this section with some 
examples.

\medskip

\begin{definition} \label{lecomplete}
A polyomino $S$ is {\it \Le-complete} if satisfies the following 
conditions:
for any $i<j$ and $k<l$, if $(j,k), (j,l), (i,l)$ are cells of $S$ then
the cell $(i,k)$ is also in $S$. Otherly said, for any three cells arranged
as a $\Le$, there is a $2\times 2$-submatrix of $S$ containing these three
cells.
\end{definition}

This is equivalent to recursive condition ``2$\times$2-connected bottom-right
CR-erasable" of \cite{AS}.
The first examples of \Le-complete polyominoes are Young diagrams in
English notation, so that this section is indeed a generalization of the
previous one. Other examples are skew shapes (also known as paralellogram 
polyominoes), as in the  right part of Figure \ref{lec}.

\begin{figure}[h!tp]
\centering
\psset{unit=8mm}
\begin{pspicture}(0,0)(4,3)
\psframe[linewidth=0.7mm](0,0)(1,1)\rput(0.5,0.5){(i,l)}
\psframe[linewidth=0.7mm](3,0)(4,1)\rput(3.5,0.5){(j,l)}
\psframe[linewidth=0.7mm](3,2)(4,3)\rput(3.5,2.5){(j,k)}
\psframe[linewidth=0.7mm,linestyle=dashed](0,2)(1,3)\rput(0.5,2.5){(i,k)}
\end{pspicture} 
\hspace{1cm}
\psset{unit=5mm}
\begin{pspicture}(0,0)(8,5)
\psline(0,5)(4,5)(4,4)(6,4)(6,3)(7,3)(7,0)(4,0)(4,2)(3,2)(3,3)(1,3)(1,4)(0,4)(0,5)
\psline(4,1)(7,1)\psline(3,2)(7,2)\psline(1,3)(7,3)\psline(0,4)(4,4)
\psline(1,5)(1,4)\psline(2,5)(2,3)\psline(3,5)(3,3)\psline(4,4)(4,2)
\psline(5,4)(5,0)\psline(6,3)(6,0)
\psline(7,2)(8,2)(8,0)(7,0)\psline(7,1)(8,1)
\psline(4,1)(3,1)(3,2)
\psframe[linewidth=0.6mm](4,0)(5,1)
\psframe[linewidth=0.6mm](6,0)(7,1)
\psframe[linewidth=0.6mm](6,2)(7,3)
\psframe[linewidth=0.6mm,linestyle=dashed](4,2)(5,3)
\end{pspicture}
\caption{\Le-completeness \label{lec}. If a \Le-complete polyomino 
contains the
three squares with continuous thick borders, it has to contain also the 
square with dashed thick borders. On the right, we have an example of skew
shape.}
\end{figure}
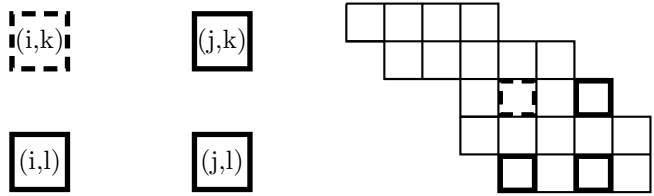

\bigskip

\subsection{The generalization of $\phi$}

We now suppose that $S$ is a young diagram, in French notation. Let $r$ be
the number of distinct parts of the corresponding partition. Let 
$c_1,\dots,c_r$ be the top right corners of $S$, sorted by growing abscissas. 
And for any top-right corner $c_i$, let $R_i$ be the rectangular subset 
of $S$ formed by all cells that are in the bottom-left quater-plane of origin 
$c_i$. See Figure \ref{rect} for an example. For any X-diagram $T$ of shape $S$
we denote by $T|_{R_i}$ the subdiagram of $T$ obtained by selecting the cells of 
$R_i$.

\begin{figure}[h!tp]
\centering\psset{unit=0.4cm}
\begin{pspicture}(0,-1)(6,5)
\psframe[fillstyle=solid, fillcolor=lightgray](0,0)(2,5)
\rput(1.5,4.5){$c_1$}
\psline(0,0)(6,0)
\psline(0,1)(6,1)
\psline(0,2)(5,2)
\psline(0,3)(5,3)
\psline(0,4)(4,4)
\psline(0,5)(2,5)
\psline(0,0)(0,5)
\psline(1,0)(1,5)
\psline(2,0)(2,5)
\psline(3,0)(3,4)
\psline(4,0)(4,4)
\psline(5,0)(5,3)
\psline(6,0)(6,1)
\rput(3,-1){$R_1$}
\end{pspicture}
\hfill
\begin{pspicture}(0,-1)(6,5)
\psframe[fillstyle=solid, fillcolor=lightgray](0,0)(4,4)
\rput(3.5,3.5){$c_2$}
\psline(0,0)(6,0)
\psline(0,1)(6,1)
\psline(0,2)(5,2)
\psline(0,3)(5,3)
\psline(0,4)(4,4)
\psline(0,5)(2,5)
\psline(0,0)(0,5)
\psline(1,0)(1,5)
\psline(2,0)(2,5)
\psline(3,0)(3,4)
\psline(4,0)(4,4)
\psline(5,0)(5,3)
\psline(6,0)(6,1)
\rput(3,-1){$R_2$}
\end{pspicture}
\hfill
\begin{pspicture}(0,-1)(6,5)
\psframe[fillstyle=solid, fillcolor=lightgray](0,0)(5,3)
\rput(4.5,2.5){$c_3$}
\psline(0,0)(6,0)
\psline(0,1)(6,1)
\psline(0,2)(5,2)
\psline(0,3)(5,3)
\psline(0,4)(4,4)
\psline(0,5)(2,5)
\psline(0,0)(0,5)
\psline(1,0)(1,5)
\psline(2,0)(2,5)
\psline(3,0)(3,4)
\psline(4,0)(4,4)
\psline(5,0)(5,3)
\psline(6,0)(6,1)
\rput(3,-1){$R_3$}
\end{pspicture}
\hfill
\begin{pspicture}(0,-1)(6,5)
\psframe[fillstyle=solid, fillcolor=lightgray](0,0)(6,1)
\rput(5.5,0.5){$c_4$}
\psline(0,0)(6,0)
\psline(0,1)(6,1)
\psline(0,2)(5,2)
\psline(0,3)(5,3)
\psline(0,4)(4,4)
\psline(0,5)(2,5)
\psline(0,0)(0,5)
\psline(1,0)(1,5)
\psline(2,0)(2,5)
\psline(3,0)(3,4)
\psline(4,0)(4,4)
\psline(5,0)(5,3)
\psline(6,0)(6,1)
\rput(3,-1){$R_4$}
\end{pspicture}
\caption{\label{rect} Example of rectangles $R_i$.}
\end{figure}

\begin{definition}\label{pivotgen}
 We call {\it pivot column} of a diagram $T$ of shape $S$ a 
column of index $j$ such that:
\begin{itemize}
\item there is a $i$, such that the $j$th column of $T|_{R_i}$ is its pivot 
  column, and such that $j$ is strictly greater than the column index of 
  $c_{i-1}$ (the last condition is automatically satisfied if $i=1$).
\item it is in righmost position among columns satisfying the first property
  (otherly said, the first property is satisfied with the greatest $i$ 
  possible).
\end{itemize}

If such a column exists, it is uniquely defined. There is no such column 
only in the case where the bottom row is a zero-row.
\end{definition}

The last affirmation is not as obvious as in the previous section, so we give
some precisions.
If there is a 1 in the bottom row, we can consider the smallest $i$ such that the 
bottom row of $T|_{R_i}$ contains a 1, and then the pivot column of $T|_{R_i}$ 
satisfies the first condition, so there is a pivot column in $T$. Reciprocally, if
there is a pivot column, it contains a 1 in bottom position so the bottom row is
not a zero-row. Hence as announced, there is no pivot column in $T$ 
only in the case where the bottom row is a zero-row.

\medskip

Now, if we have an X-diagram $T$ the mixed diagram $\phi(T)$ is defined as in the previous 
case by a column-by-column transformation, but with one difference. Indeed, in the previous 
case a column identical to the pivot column was transformed into a column having a single 1
in bottom position. Now, the modified columns may have different size, and the condition is
the following. If a column to the right of the pivot column, is identical to the pivot column
except that some top cells are removed, it is transformed into a column having a single 1 in
bottom position. This is illustrated in the following example:

\[
\psset{unit=0.4cm}
\begin{pspicture}(0,0)(9,6)
\psline(0,0)(9,0)
\psline(0,1)(9,1)
\psline(0,2)(9,2)
\psline(0,3)(8,3)
\psline(0,4)(7,4)
\psline(0,5)(4,5)
\psline(0,6)(2,6)
\psline(0,0)(0,6)
\psline(1,0)(1,6)
\psline(2,0)(2,6)
\psline(3,0)(3,5)
\psline(4,0)(4,5)
\psline(5,0)(5,4)
\psline(6,0)(6,4)
\psline(7,0)(7,4)
\psline(8,0)(8,3)
\psline(9,0)(9,2)
\rput(0.5,0.5){1}\rput(0.5,1.5){1}\rput(0.5,2.5){1}\rput(0.5,3.5){1}\rput(0.5,4.5){1}\rput(0.5,5.5){1}
\rput(1.5,0.5){0}\rput(1.5,1.5){1}\rput(1.5,2.5){0}\rput(1.5,3.5){1}\rput(1.5,4.5){0}\rput(1.5,5.5){1}
\rput(2.5,0.5){1}\rput(2.5,1.5){1}\rput(2.5,2.5){1}\rput(2.5,3.5){1}\rput(2.5,4.5){0}
\rput(3.5,0.5){\bf 1}\rput(3.5,1.5){\bf 1}\rput(3.5,2.5){\bf 0}\rput(3.5,3.5){\bf 1}\rput(3.5,4.5){\bf 0}
\rput(4.5,0.5){0}\rput(4.5,1.5){1}\rput(4.5,2.5){0}\rput(4.5,3.5){0}
\rput(5.5,0.5){1}\rput(5.5,1.5){1}\rput(5.5,2.5){0}\rput(5.5,3.5){1}
\rput(6.5,0.5){0}\rput(6.5,1.5){1}\rput(6.5,2.5){0}\rput(6.5,3.5){1}
\rput(7.5,0.5){0}\rput(7.5,1.5){1}\rput(7.5,2.5){0}
\rput(8.5,0.5){1}\rput(8.5,1.5){1}
\rput(-1.5,2.5){$T=$}
\end{pspicture}\hspace{1mm},\hspace{2cm}
\begin{pspicture}(0,0)(9,6)
\psline(0,0)(9,0)
\psline(0,1)(9,1)
\psline(0,2)(9,2)
\psline(0,3)(8,3)
\psline(0,4)(7,4)
\psline(0,5)(4,5)
\psline(0,6)(2,6)
\psline(0,0)(0,6)
\psline(1,0)(1,6)
\psline(2,0)(2,6)
\psline(3,0)(3,5)
\psline(4,0)(4,5)
\psline(5,0)(5,4)
\psline(6,0)(6,4)
\psline(7,0)(7,4)
\psline(8,0)(8,3)
\psline(9,0)(9,2)
\rput(0.5,0.5){0}\rput(0.5,1.5){1}\rput(0.5,2.5){1}\rput(0.5,3.5){1}\rput(0.5,4.5){1}\rput(0.5,5.5){1}
\rput(1.5,0.5){0}\rput(1.5,1.5){1}\rput(1.5,2.5){0}\rput(1.5,3.5){1}\rput(1.5,4.5){0}\rput(1.5,5.5){1}
\rput(2.5,0.5){0}\rput(2.5,1.5){1}\rput(2.5,2.5){1}\rput(2.5,3.5){1}\rput(2.5,4.5){0}
\rput(3.5,0.5){\bf 1}\rput(3.5,1.5){\bf 1}\rput(3.5,2.5){\bf 0}\rput(3.5,3.5){\bf 1}\rput(3.5,4.5){\bf 0}
\rput(4.5,0.5){1}\rput(4.5,1.5){1}\rput(4.5,2.5){0}\rput(4.5,3.5){0}
\rput(5.5,0.5){1}\rput(5.5,1.5){0}\rput(5.5,2.5){0}\rput(5.5,3.5){0}
\rput(6.5,0.5){1}\rput(6.5,1.5){1}\rput(6.5,2.5){0}\rput(6.5,3.5){1}
\rput(7.5,0.5){1}\rput(7.5,1.5){1}\rput(7.5,2.5){0}
\rput(8.5,0.5){1}\rput(8.5,1.5){0}
\rput(-1.7,2.5){$\phi(T)=$}
\end{pspicture}\hspace{1mm}.
\]
Here, the respective pivot columns of $T|_{R_1},\dots,T|_{R_5}$ have indices 
1, 4, 4, 4, and 1. The indices $i$ satisfying the first condition of 
Definition \ref{pivotgen} are 1 and 2. So the pivot column of $R$ is the 
one of $T_{R_2}$, {\it i.e.} the 4th one. Now, we can see that the 6th
and 9th column of $T$ are identical to the pivot column, except that some
top cells are removed. So the 6th column of $\phi(T)$, and also the 9th column, 
has a single 1 in bottom position.

\medskip

In the previous section we saw that two differently-filled columns of $T$ give rise 
to two differently-filled columns in $\phi(T)$, provided they are both strictly to 
the left (or to the right) of the pivot column. This is also true in the present case.
The main point is the following lemma, where we really use Definition 
\ref{pivotgen}.

\begin{lem} Let $j$ be the index of the pivot column of $T$. Let $k$ and $\ell$ be
such that $j<k<\ell$, the $k$th column and the $\ell$th column have the same size, and
they are identical except that their bottom entries are different. Then one of 
the $k$th and $\ell$th columns is identical to the pivot column (except that some
top cells are removed).
\end{lem}
For example the 6th and 7th columns in the previous example $T$ satisfy this 
condition and indeed the 6th column is identical to the pivot column, except 
that the top cell is removed.

\begin{proof} Notice that if the $k$th and $\ell$th columns have the same size as
the pivot column then it is just a consequence of Lemma \ref{lem4}. Now we 
suppose that the $k$th and $\ell$th columns are strictly smaller than the pivot 
column. Let $i$ be the size of the $k$th and $\ell$th columns. We distinguish two cases.
\begin{itemize}
\item We suppose first that the $i-1$ top cells of the $k$th column contain strictly 
  more 1s than the corresponding $i-1$ cells in the pivot column. Then, the bottom entry
  of the $k$th and $\ell$th columns should be a 1, otherwise there would be an occurrence
  of \pattern 0110. This contradicts the fact that these entries are different.
\item We suppose then that the $i-1$ top cells of the $k$th column contain strictly 
  more 0s than the corresponding $i-1$ cells in the pivot column (say this number of 
  0s is $x$ ). Then, let $h$ be such that the $k$th and $\ell$th columns are columns of 
  the rectangles $R_h$. The pivot column of $T|_{R_h}$ contains at least $x$ 0s.
  Because of the second point in the Definition \ref{pivotgen}, it is possible to 
  find in $T|_{R_{h-1}}$ a column having a 1 in bottom position and at least $x$ 0s 
  in the $i$ bottom entries. If this column is also the pivot column of $T$ there is a
  contradiction, otherwise we can find a column $T|_{R_{h-2}}$ having a 1 in bottom 
  position and at least $x$ 0s in the $i$ bottom entries, so that after some steps
  we eventually obtain a contradiction.
\end{itemize}
This shows that the $i-1$ top entries of the $k$th or $\ell$th column contain as many 0s 
as the corresponding entries in the pivot column. One of these two columns has
a 1 in bottom position, and this one is identical to the $i$ bottommost entries of the
pivot column.
\end{proof}

Apart from the situation in the previous lemma, we can see that the bottom entry
of any column is determined by the other entries and the fact that this column is
to the left or to the right of the pivot column. This is similar to the Lemmas
\ref{lem3} and \ref{lem4}. So we have proved:

\begin{prop} Proposition \ref{phi} remains true when the shape of the diagrams
 is a Young diagram in French notation.
\end{prop}

\bigskip

The last step of this section is to notice that the bijection $\phi$
for Young diagrams in French notation has an immediate generalization
to other polyominoes: they are the ones obtained by permuting the
$k-1$ top rows of a Young diagram in French notation with $k$ rows.

In this case we also define rectangular subsets $R_i^{\phantom i}$ by  
the following construction.
Let $r$ be the number of distinct columns of $S$, and $C_1,\dots,C_r$ be
the columns, sorted by growing indices, such that each of them is distinct from 
all columns of greater index. And let $R_i$ be the set of cells that are to 
the left (in the same row) of a cell of $C_i$. For an example see Figure 
\ref{rect2}.

There is also a bijection between X-diagrams and mixed diagrams of shape $S$, 
preserving any zero-row or zero-column. It is defined exactly the same way as in
the case where $S$ in a Young diagram in French notation.

\begin{figure}[h!tp]
\centering\psset{unit=0.4cm}
\begin{pspicture}(0,-1)(6,5)
\psframe[fillstyle=solid, fillcolor=lightgray](0,0)(2,5)
\psline(0,0)(6,0)
\psline(0,1)(6,1)
\psline(0,2)(5,2)
\psline(0,3)(5,3)
\psline(0,4)(5,4)
\psline(0,5)(5,5)
\psline(0,0)(0,5)
\psline(1,0)(1,5)
\psline(2,0)(2,5)
\psline(3,0)(3,3)\psline(3,4)(3,5)
\psline(4,0)(4,3)\psline(4,4)(4,5)
\psline(5,0)(5,1)\psline(5,2)(5,3)\psline(5,4)(5,5)
\psline(6,0)(6,1)
\rput(3,-1){$R_1$}
\end{pspicture}
\hfill
\begin{pspicture}(0,-1)(6,5)
\psframe[fillstyle=solid, fillcolor=lightgray](0,0)(4,3)
\psframe[fillstyle=solid, fillcolor=lightgray](0,4)(4,5)
\psline(0,0)(6,0)
\psline(0,1)(6,1)
\psline(0,2)(5,2)
\psline(0,3)(5,3)
\psline(0,4)(5,4)
\psline(0,5)(5,5)
\psline(0,0)(0,5)
\psline(1,0)(1,5)
\psline(2,0)(2,5)
\psline(3,0)(3,3)\psline(3,4)(3,5)
\psline(4,0)(4,3)\psline(4,4)(4,5)
\psline(5,0)(5,1)\psline(5,2)(5,3)\psline(5,4)(5,5)
\psline(6,0)(6,1)
\rput(3,-1){$R_2$}
\end{pspicture}
\hfill
\begin{pspicture}(0,-1)(6,5)
\psframe[fillstyle=solid, fillcolor=lightgray](0,0)(5,1)
\psframe[fillstyle=solid, fillcolor=lightgray](0,2)(5,3)
\psframe[fillstyle=solid, fillcolor=lightgray](0,4)(5,5)
\psline(0,0)(6,0)
\psline(0,1)(6,1)
\psline(0,2)(5,2)
\psline(0,3)(5,3)
\psline(0,4)(5,4)
\psline(0,5)(5,5)
\psline(0,0)(0,5)
\psline(1,0)(1,5)
\psline(2,0)(2,5)
\psline(3,0)(3,3)\psline(3,4)(3,5)
\psline(4,0)(4,3)\psline(4,4)(4,5)
\psline(5,0)(5,1)\psline(5,2)(5,3)\psline(5,4)(5,5)
\psline(6,0)(6,1)
\rput(3,-1){$R_3$}
\end{pspicture}
\hfill
\begin{pspicture}(0,-1)(6,5)
\psframe[fillstyle=solid, fillcolor=lightgray](0,0)(6,1)
\psline(0,0)(6,0)
\psline(0,1)(6,1)
\psline(0,2)(5,2)
\psline(0,3)(5,3)
\psline(0,4)(5,4)
\psline(0,5)(5,5)
\psline(0,0)(0,5)
\psline(1,0)(1,5)
\psline(2,0)(2,5)
\psline(3,0)(3,3)\psline(3,4)(3,5)
\psline(4,0)(4,3)\psline(4,4)(4,5)
\psline(5,0)(5,1)\psline(5,2)(5,3)\psline(5,4)(5,5)
\psline(6,0)(6,1)
\rput(3,-1){$R_4$}
\end{pspicture}
\caption{\label{rect2} Example of the rectangles $R_i$. The columns $C_i$ have indices
2,4,5, and 6.}
\end{figure}
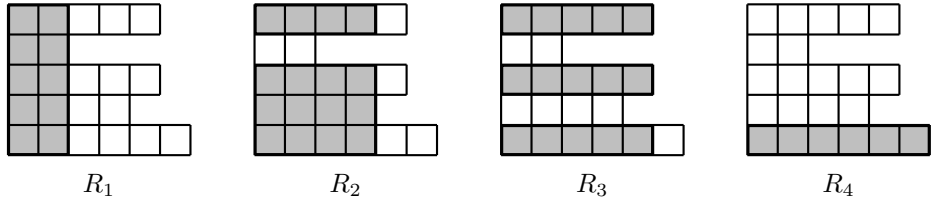

\bigskip

\subsection{The generalization of $\Phi$}

As in the previous case, $\Phi$ is recursively defined with respect to the number 
of rows of the diagrams. The only thing to check is that the 
polyominoes we can obtain are precisely the \Le-complete ones.

\begin{lem} For any polyomino $S$ the following conditions are equivalent:
\begin{itemize}
\item Every column of $S$ intersects the bottom row, and $S$ is $\Le$-complete.
\item The polyomino $S$ is obtained by permuting the $k-1$ top rows, of 
a Young diagram in French notation with $k$ rows.
\end{itemize}
\end{lem}

\begin{proof} An example is given in Figure \ref{rect2}.
Suppose for example that the first condition is true. Then the bottom
row is a set of contiguous cells, and it is of maximal size. Since the 
polyomino is
\Le-complete, we obtain that each other row is also a set of contiguous cells and
has the same left extremity as the bottom row. This implies the second condition.
The converse is obtained as easily.
\end{proof}

For any \Le-complete polyomino $S$, we can consider the subset $S'$ of 
all cells
that are above a cell of the bottom row. Then $S'$ satisfies the condition of the
previous Lemma. Then for any X-diagram $T$ of shape $S$, we define $\phi(T)$ by 
replacing $T|_{S'}$ with its image by $\phi$. And $\Phi(T)$ is defined recursively as
in the previous case, by replacing the $k-1$ top rows of $\phi(T)$ with their image
by $\Phi$. Although the details may seem messy because of the more general 
polyominoes, the ideas are really the same as in the previous case.

\begin{prop} Proposition \ref{pphi} remains true with diagrams whose 
shape is a \Le-complete polyomino.
\end{prop}

An example of a \Le-complete polyomino $S$, the subset $S'$ and the 
rectangles $R_i$ is given in Figure \ref{ex}.

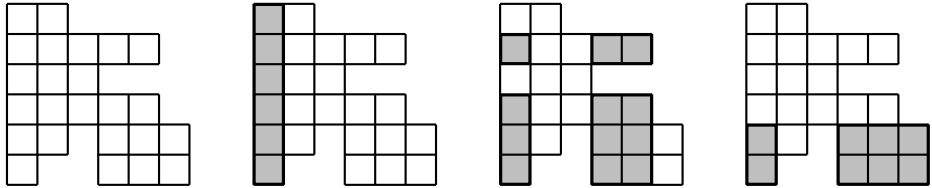
\begin{figure}[h!tp]
\centering\psset{unit=0.4cm}
\begin{pspicture}(0,0)(6,6)
\psline(0,0)(0,6)
\psline(1,0)(1,6)
\psline(2,1)(2,6)
\psline(3,0)(3,5)
\psline(4,0)(4,3)\psline(4,4)(4,5)
\psline(5,0)(5,3)\psline(5,4)(5,5)
\psline(6,0)(6,2)
\psline(0,0)(1,0)\psline(3,0)(6,0)
\psline(0,1)(2,1)\psline(3,1)(6,1)
\psline(0,2)(6,2)
\psline(0,3)(5,3)
\psline(0,4)(5,4)
\psline(0,5)(5,5)
\psline(0,6)(2,6)
\end{pspicture}
\hfill
\begin{pspicture}(0,0)(6,6)
\psframe[fillstyle=solid, fillcolor=lightgray](0,0)(1,6)
\psline(0,0)(0,6)
\psline(1,0)(1,6)
\psline(2,1)(2,6)
\psline(3,0)(3,5)
\psline(4,0)(4,3)\psline(4,4)(4,5)
\psline(5,0)(5,3)\psline(5,4)(5,5)
\psline(6,0)(6,2)
\psline(0,0)(1,0)\psline(3,0)(6,0)
\psline(0,1)(2,1)\psline(3,1)(6,1)
\psline(0,2)(6,2)
\psline(0,3)(5,3)
\psline(0,4)(5,4)
\psline(0,5)(5,5)
\psline(0,6)(2,6)
\end{pspicture}
\hfill
\begin{pspicture}(0,0)(6,6)
\psframe[fillstyle=solid, fillcolor=lightgray](0,0)(1,3)
\psframe[fillstyle=solid, fillcolor=lightgray](0,4)(1,5)
\psframe[fillstyle=solid, fillcolor=lightgray](3,0)(5,3)
\psframe[fillstyle=solid, fillcolor=lightgray](3,4)(5,5)
\psline(0,0)(0,6)
\psline(1,0)(1,6)
\psline(2,1)(2,6)
\psline(3,0)(3,5)
\psline(4,0)(4,3)\psline(4,4)(4,5)
\psline(5,0)(5,3)\psline(5,4)(5,5)
\psline(6,0)(6,2)
\psline(0,0)(1,0)\psline(3,0)(6,0)
\psline(0,1)(2,1)\psline(3,1)(6,1)
\psline(0,2)(6,2)
\psline(0,3)(5,3)
\psline(0,4)(5,4)
\psline(0,5)(5,5)
\psline(0,6)(2,6)
\end{pspicture}
\hfill
\begin{pspicture}(0,0)(6,6)
\psframe[fillstyle=solid, fillcolor=lightgray](0,0)(1,2)
\psframe[fillstyle=solid, fillcolor=lightgray](3,0)(6,2)
\psline(0,0)(0,6)
\psline(1,0)(1,6)
\psline(2,1)(2,6)
\psline(3,0)(3,5)
\psline(4,0)(4,3)\psline(4,4)(4,5)
\psline(5,0)(5,3)\psline(5,4)(5,5)
\psline(6,0)(6,2)
\psline(0,0)(1,0)\psline(3,0)(6,0)
\psline(0,1)(2,1)\psline(3,1)(6,1)
\psline(0,2)(6,2)
\psline(0,3)(5,3)
\psline(0,4)(5,4)
\psline(0,5)(5,5)
\psline(0,6)(2,6)
\end{pspicture}
\caption{\label{ex} Example of a \Le-complete polyomino $S$, and 
rectangles $R_i$. The
subset $S'$ is obtained by removing the second and third column.}
\end{figure}

\bigskip

\section{Overview of related results and conclusion}

Many questions may arise about how to generalize this work, as can be seen in
\cite{AS}. We can consider many kinds of avoided patterns, and each of them raises
the question of which class of polyominoes have interesting properties. 
In this section, we sketch some consequences of the results of the 
previous sections and other related results.

\bigskip

First we show that, as announced in the introduction, we can obtain a bijection
between X-diagrams and \Le-diagrams preserving the set of zero-columns and the
set of unrestricted rows. This is done by the following modification of the definition 
of the pivot column:

\begin{definition}
We call {\it pivot column} of an X-diagram $T$ a 
non-zero column among the $\lambda_k$ first ones (otherly said, it is a column of
maximal size) such that:
\begin{itemize}
\item there is a 0 in bottom position,
\item it has a maximal number of 1s among columns satisfying the previous 
  property,
\item it is in leftmost position among columns satisfying the two previous 
  properties.
\end{itemize}

If such a column exists, it is uniquely defined (by the third property).
There is no such column only in the case where the bottom row is unrestricted.
\end{definition}

With this new definition, we can define new bijections $\phi_2$ and $\Phi_2$
the same way we defined $\phi$ and $\Phi$, and we obtain the following:

\begin{prop} For any \Le-complete polyomino $S$, the map $\Phi_2$ is a 
bijection between
X-diagrams of shape $S$ and \Le-diagrams of shape $S$, and this bijection preserves
the zero-columns and the restricted rows.
\end{prop}

Now we consider a column-convex polyomino $S$ such that the top of each column are at the 
same height, as in Figure \ref{polyo}. Such objects are sometimes called {\it stalactite
polyominoes}. It is \Le-complete. If we have an X-diagram of 
shape $S$, we can permute the columns of the polyomino and get another X-diagram. It is 
a consequence of the symmetry in the pattern pair (\pattern 1001, \pattern 0110). So with
the previous bijection, we obtain that the number of \Le-diagram of shape $S$ only depends
on the column lengths. 

\begin{figure}[h!tp]
\centering
\psset{unit=4mm}\begin{pspicture}(0,1)(8,6)
\psline(0,6)(0,2)(1,2)(1,1)(2,1)(2,3)(3,3)(3,4)(4,4)(4,3)(5,3)(5,3)(6,3)(6,4)(7,4)(7,2)(8,2)(8,6)(0,6)
\psline(0,5)(8,5)
\psline(0,4)(3,4)\psline(4,4)(6,4)\psline(7,4)(8,4)
\psline(0,3)(3,3)\psline(7,3)(8,3)
\psline(1,2)(2,2)
\psline(1,1)(1,6)\psline(2,2)(2,6)\psline(3,3)(3,6)
\psline(4,4)(4,6)\psline(5,3)(5,6)\psline(6,4)(6,6)\psline(7,4)(7,6)
\end{pspicture}
\caption{\label{polyo} Example of a stalactite polyomino.}
\end{figure}

It is interesting to notice that this is very close to some results of \cite{JJ}
(we thank Mireille Bousquet-M\'elou for mentioning this reference). Indeed,
one of the objects studied in this reference are the fillings of stalactite polyominoes
by 0s and 1s, avoiding a pattern which is the identity matrix of a given size, and with
a maximal number of 1s. Jonsson shows that the number of such fillings only depends of the
column length of the polyomino, which is a similarity with the case of \Le-diagrams.


\bigskip

\end{document}